\documentclass[11pt]{article}

\usepackage{amsmath,amssymb,amsthm,amscd,epsfig}
\newcommand{\vp}{\varepsilon}

\theoremstyle{plain}

\newtheorem{thm}{Theorem}

\newtheorem{cor}{Corollary}

\theoremstyle{definition}

\newtheorem{defn}{Definition}

\theoremstyle{remark}

\begin{document}

\title{Entropy production and folding of the phase space in chaotic dynamics}

\author{Eugen Mihailescu}
\date{}
\maketitle

\begin{abstract}
We study the entropy production of Gibbs (equilibrium) measures
for chaotic dynamical systems with folding of the phase space. The
dynamical chaotic model is that generated by a hyperbolic
non-invertible map $f$ on a general basic (possibly fractal) set
$\Lambda$; the non-invertibility creates new phenomena and
techniques than in the diffeomorphism case. We prove a formula for
the \textit{entropy production}, involving an asymptotic
logarithmic degree, with respect to the equilibrium measure
$\mu_\phi$ associated to the potential $\phi$. This formula helps
us calculate the entropy production of the measure of maximal
entropy of $f$. Next for hyperbolic toral endomorphisms, we prove
that all Gibbs states $\mu_\phi$ have \textit{non-positive entropy
production} $e_f(\mu_\phi)$. We study also the entropy production
of the \textit{inverse Sinai-Ruelle-Bowen measure} $\mu^-$ and
show that for a large family of maps, it is \textit{strictly
negative}, while at the same time the entropy production of the
respective (forward) Sinai-Ruelle-Bowen measure $\mu^+$ is
strictly positive.
\end{abstract}



\section{Entropy production. Outline of main results.}

In statistical mechanics, one concerns himself with the
\textit{stationary (steady) states}, which are probability
measures on the phase space, invariant under time evolution. The
study of such states can be done with the help of dynamical
systems and ergodic theory (for instance \cite{Bo}, \cite{DSS},
\cite{ER}, \cite{Ru-survey99}, \cite{Ru-folding}, \cite{S}, etc.)
The fundamental postulate of statistical mechanics (see
\cite{DSS}) says that a physical system in thermodynamical
equilibrium is described by Gibbs measures. In the nonequilibrium
scenario, once a system is kept out of equilibrium and subjected
to non-Hamiltonian forces, the energy is in general not conserved
(\cite{Ru-survey99}); so we couple it with a large external system
called a thermostat, and record the entropy changes. It is thus
justified to have a notion of \textit{entropy production}, as a
measure of the average differences in the entropy of the system
over time. Certain nonequilibrium steady states are described also
by Gibbs states, but for a different problem (see
\cite{Ru-survey99}). \ From a mathematical point of view, Ruelle
identifies in \cite{Ru-folding} (see also \cite{Ru-97} and
\cite{Ru-survey99}) several types of entropy productions given by:
i) a diffeomorphism $f$ of a manifold $M$; ii) an endomorphism $f$
on $M$, i.e a non-invertible smooth map $f$; here the folding of
$M$ by $f$ will itself contribute to the entropy production; iii)
a diffusion model given by a map $f$ restricted to a neighbourhood
of a compact invariant set $X \subset M$. Whether entropy
production of a state is positive or not, is not clear a priori.

In this paper we are concerned with the case when $f$ is a \textit{smooth endomorphism} on a manifold $M$, having a compact
invariant set $\Lambda \subset M$. Since hyperbolicity plays an important role in modelling time evolutions in statistical
 mechanics (for instance \cite{Ru-survey99}, \cite{Ru-folding}, etc.), we shall assume that the endomorphism $f$ is
 \textit{hyperbolic} in the sense of \cite{Ru-carte89}, i.e that there exists a continuous splitting of the tangent
 bundle over the inverse limit $\hat \Lambda$ ($\hat \Lambda$ being the space of past trajectories of points in
 $\Lambda$), into stable and unstable directions; $f$ is not assumed expanding. This implies that we have stable directions and local stable manifolds of
 type $W^s_r(x), x \in \Lambda$, and unstable directions and local unstable manifolds of type $W^u_r(\hat x),
 \hat x \in \hat \Lambda$. Thus through a given point $x$ there may pass many (even uncountably many) local unstable
  manifolds corresponding to different prehistories of $x$ in $\hat \Lambda$. This follows from the
   non-invertibility of $f$.

For systems given by Anosov diffeomorphisms, or for diffeomorphisms having a hyperbolic attractor, we have the existence of
Sinai-Ruelle-Bowen (SRB) measures, which are natural invariant measures in the sense that they describe the distribution of trajectories of Lebesgue-almost all points in a neighbourhood of the attractor. As was shown by Sinai, the
SRB measure of an Anosov diffeomorphism $f$, is in fact the Gibbs state of a Holder potential $\Phi^u$, where
$\Phi^u = -\log |\text{det} Df_u|$ (the \textit{unstable potential}).

For a $\mathcal{C}^2$ diffeomorphism $f$, the \textit{entropy
production} of an arbitrary $f$-invariant probability measure
$\mu$ is defined (see \cite{Ru-folding}) as $ e_f(\mu) = -\int
\log |\text{det}(Df)(x)| d\mu(x).$  If $\mu$ is an SRB state, then
Ruelle proved in \cite{Ru-folding} that $e_f(\mu) \ge 0$; moreover
if the SRB state $\mu$ has no vanishing Lyapunov exponents and if
$e_f(\mu) = 0$, then $\mu$ must be absolutely continuous with
respect to the Lebesgue measure (see \cite{Ru-folding},
\cite{Ru-survey99}; and \cite{L} for a characterization of
invariant absolutely continuous measures). SRB measures exist also
for diffeomorphisms having Axiom A attractors, as shown by
Ruelle (see for instance \cite{ER}). Moreover SRB measures do exist also for
Anosov endomorphisms and for endomorphisms with hyperbolic
attractors, and they are equal to the equilibrium measures of the respective unstable potentials on the inverse limit spaces (see \cite{QZ}).

For a \textit{non-invertible} smooth map $f$ on a Riemannian
manifold $M$ and an $f$-invariant probability $\mu$ on $M$,
Ruelle defined in \cite{Ru-folding} the \textit{entropy
production} of $\mu$ by:
\begin{equation}\label{entprod}
e_f(\mu) := F_f(\mu) - \int \log |\text{det}(Df)(x)|d\mu(x),
\end{equation}
where $F_f(\mu)$ is called the \textit{folding entropy} of $\mu$
with respect to $f$. $F_f(\mu)$ is defined as the conditional
entropy $H_\mu(\epsilon|f^{-1}\epsilon)$, of $\epsilon$ with
respect to $f^{-1}\epsilon$, where $\epsilon$ is the single point
partition.

For example we can obtain stationary measures
$\mu$, with respect to $f$, as weak limits (when $n \to \infty$) of averages of type
\begin{equation}\label{rho-n}
\frac 1n \mathop{\Sigma}\limits_{k = 0}^{n-1} f^k \rho,
\end{equation}
  where $\rho$ is an absolutely continuous probability with respect to Lebesgue measure,  with density $\bar \rho$.
  For such $f$-invariant limit measures $\mu$, Ruelle showed that the
   entropy production is non-negative (\cite{Ru-folding}).
There do exist in fact dynamical systems from physics presenting
non-invertibility (see for example \cite{ER}, \cite{Ru-folding},
\cite{Ru-survey99}). Hyperbolicity also appears to have physical
meaning and it may be used as an approximation for certain
physical phenomena (see for example \cite{Ru-survey99}). The study
of dynamics of hyperbolic endomorphisms, and of their
characteristics different from diffeomorphisms, appeared also in
\cite{Bot}, \cite{Ru-carte89}, \cite{M-MZ},
 \cite{M-Cam}, \cite{QS}, etc.

From the above, it is then justified to study the entropy
production of equilibrium measures of Holder potentials for
non-invertible smooth maps $f$ on basic sets $\Lambda$ on which
$f$ is hyperbolic and transitive; here by \textit{basic set} (or
locally maximal set \cite{KH}) we mean a compact $f$-invariant set
$\Lambda$ s.t $\Lambda = \mathop{\cap}\limits_{n \in \mathbb Z}
f^n(U)$, for a neighbourhood $U$ of $\Lambda$. Our endomorphism
$f$ is not assumed expanding. \ The \textbf{main results} of the
paper are the following:

In \textbf{Theorem \ref{Jacobian}} we give a precise estimate for
the Jacobian (in the sense of Parry, \cite{Pa}) of the equilibrium
measure $\mu_\phi$ associated to an arbitrary Holder potential
$\phi$, with respect to the iterate $f^n$. This estimate is
independent of $n$ and will allow us to express the
\textbf{folding entropy} of $\mu_\phi$ with respect to $f$. Next
we  will describe the folding entropy of $\mu_\phi$ as the limit
of the weighted integral, of the logarithm of the degree function
of $f^n$ with respect to $\mu_\phi$ on $\Lambda$. In this way in
\textbf{Theorem \ref{folding-ent}} we give a formula for the
entropy production of $\mu_\phi$ in terms of an "asymptotic
logarithmic degree" (with respect to $\mu_\phi$) minus the
integral of the Jacobian with respect to the Riemannian metric;
the asymptotic logarithmic degree takes into consideration only
those $n$-preimages (i.e preimages
 with respect to $f^n$) which behave well with respect to $\phi$.
 In \textbf{Corollary \ref{deg-ent}} we will use the formula proved in Theorem
 \ref{folding-ent} in order to calculate the folding entropy
 of the measure of maximal entropy.

We investigate next the case of a hyperbolic toral endomorphism on
$\mathbb T^k$ and its Gibbs measures associated to various Holder
potentials. We prove in \textbf{Corollary \ref{neg-ent}} that in
this setting, the entropy production of \textit{any} equilibrium
measure of a Holder potential is \textbf{non-positive}.

In \cite{M-JSP}, we introduced an \textbf{inverse SRB measure}
$\mu^-$ which has physical relevance since it gives the
distribution of past trajectories with respect to the endomorphism
$f$,  for Lebesgue almost all points in a neighbourhood of a
hyperbolic repellor. This unique inverse SRB measure is not just
the SRB measure for $f^{-1}$, since our map $f$ is non-invertible
in general. We proved that in fact $\mu^-$ is the equilibrium
measure of the stable potential (with respect to the
\textit{forward} system), and that it is the only invariant
probability having absolutely continuous conditional measures on
local stable manifolds. \ Here we will show in \textbf{Theorem
\ref{ent-tor}} that for perturbations of hyperbolic toral
endomorphisms,  the entropy production of $\mu^-$ is strictly
negative unless $\mu^-$ is equal to the (forward) SRB measure
$\mu^+$ of the endomorphism $f$, in which case both are absolutely continuous.  In \textbf{Corollary \ref{exp}
a)} we show that most maps in a neighbourhood of a hyperbolic
toral endomorphism have inverse SRB measures with negative entropy
production. And we actually construct in Corollary \ref{exp} b) a
family of perturbations of hyperbolic toral endomorphisms, whose
respective inverse SRB measures have \textbf{negative entropy
production}. In particular an endomorphism with negative entropy
production, will not be a stationary measure obtained as a weak
limit of averages of iterates of absolutely continuous measures as
in (\ref{rho-n}). In this way we find certain chaotic (hyperbolic)
systems with folding of phase space, and Gibbs states for them
having negative entropy production.

Several interesting and important results from statistical physics
point towards the profound relationship between entropy production
and the time arrow/irreversibility, and also the possibility of
negative entropy production on short time scales (for exp.
\cite{DSS}, \cite{ES}, \cite{Leb}, \cite{Ru-survey99}, \cite{W},
etc.) However our results are abstract mathematical ones, and we do not
investigate here possible physical implications, if any.

\section{Main results and proofs.}

For the rest of the paper let us fix a smooth (say
$\mathcal{C}^2$) non-invertible map $f:M \to M$ defined on a
compact Riemannian manifold and let $\Lambda$ be a fixed basic set
of $f$, i.e there exists some neighbourhood $U$ of $\Lambda$ with
$\Lambda = \mathop{\cap}\limits_{n \in \mathbb Z} f^n(U)$. Assume
also that $f$ is \textbf{transitive and hyperbolic} on $\Lambda$.
Sometimes the set $\Lambda$ may be the whole manifold as in the
case of Anosov endomorphisms (for example for hyperbolic toral
endomorphisms). However in general $\Lambda$ may not be totally
invariant, i.e we do not always have $f^{-1}(\Lambda) = \Lambda$.

Here hyperbolicity is understood in the sense of
\textit{endomorphisms} (i.e non-invertible maps) (see
\cite{Ru-carte89}), i.e there exists a continuous splitting of the
tangent bundle into stable and unstable directions, over the
inverse limit $\hat \Lambda$ consisting of sequences of
consecutive preimages,  $$\hat \Lambda = \{\hat x = (x,
x_{-1}, x_{-2}, \ldots,) \ \text{with} \ x_{-i} \in \Lambda,
f(x_{-i}) = x_{-i+1}, i \ge 1\}$$ For any element $\hat x = (x,
x_{-1}, x_{-2}, \ldots) \in \hat \Lambda$ we have a stable
direction $E^s_x$ (which depends only on $x$) and an unstable
direction $E^u_{\hat x}$. Consequently there exists a small $r >0$
so that we can construct local stable and local unstable
manifolds, $W^s_r(x)$ and $W^u_r(\hat x)$ for any $\hat x \in \hat
\Lambda$. We shall also denote
\begin{equation}\label{DSU}
Df_s(x) := Df|_{E^s_x}, \ x \in \Lambda \ \text{and} \ Df_u(\hat
x) := Df|_{E^u_{\hat x}}, \ \hat x \in \hat \Lambda
\end{equation}

The endomorphism $f$ is assumed to have stable directions too, so it is non-expanding. \ More about hyperbolicity for endomorphisms can be found in
\cite{Ru-carte89}, \cite{M-DCDS06}, etc. When the map is not
invertible, there appear significantly different phenomena and
different techniques than in the case of diffeomorphisms (as for
example in \cite{Bot}, \cite{Ru-survey99}, \cite{M-MZ}, \cite{M-Cam}, etc.)

 We will use in the sequel the notions of
\textit{Jacobian of an invariant measure} introduced by Parry in
\cite{Pa}. Let $f:M \to M$ be a smooth endomorphism on the manifold
$M$ and $\mu$ an $f$-invariant probability on $M$ (whose support
may be smaller than $M$); assume also that $f$ is at most
countable-to-one. Then as shown by Rohlin (\cite{Ro}, \cite{Pa}),
there exists a measurable partition $\xi = (A_0, A_1, \ldots)$ so
that $f$ is injective on each $A_i$. It was proved that the
push-forward measure $((f|_{A_i})^{-1})_*\mu$ is absolutely
continuous on $A_i$ with respect to $\mu$; so it makes sense to
define (as in \cite{Pa}) the respective Radon-Nykodim derivative,
which will be called the \textbf{Jacobian} of $\mu$ with respect
to $f$:

$$ J_f(\mu)(x) = \frac{d \mu \circ (f|_{A_i})}{\mu} (x), \
\mu-\text{a.e on} \ A_i, i \ge 0$$

Notice that $J_f(\mu)(x) \ge 1, \mu-\text{a.e} \ x$. We have also
a Chain Rule when dealing with a composition of maps, namely
$$J_{f\circ g}(\mu) = J_f(g_*\mu) J_g(\mu)$$

\begin{defn}\label{comparable}
Given two positive quantities $Q_1(n, x), Q_2(n, x)$, we will say
that they are \textbf{comparable} if there exists a positive
constant $C$ so that $\frac 1C \le \frac{Q_1(n, x)}{Q_2(n, x)} \le
C$ for all $n, x$.
\end{defn}

Recall also (for example from \cite{KH}) that, given an expansive
homeomorphism $f:X \to X$ on a compact metric space, having the
specification property, the equilibrium measure $\mu_\phi$ of the
Holder potential $\phi$ satisfies $A_\vp e^{S_n\phi(x) - nP(\phi)}
\le \mu_\phi(B_n(x, \vp)) \le B_\vp e^{S_n\phi(x)-nP(\phi)}$,
where $B_n(x, \vp):= \{y \in X, d(f^i y, f^i x) < \vp, i - 0,
\ldots, n-1\}$, $P(\phi)$ denotes the topological pressure of
$\phi$ with respect to $f$, and where the positive constants
$A_\vp, B_\vp$ are independent of $x, n$. The general
homeomorphism framework above allows us to apply this result to
equilibrium measures on the inverse limit $\hat \Lambda$. If $\pi:
\hat \Lambda \to \Lambda, \pi(\hat x):= x, \hat x \in \hat
\Lambda$ is the \textit{canonical projection} and if $\phi$ is a
Holder potential on $\Lambda$, then $\mu_\phi$ is the unique
equilibrium measure for $\phi$ on $\Lambda$ if and only if
$$\mu_\phi = \pi_*\mu_{\phi\circ\pi},$$ where $\mu_{\phi\circ\pi}$
is the unique equilibrium measure of $\phi\circ\pi$ on the compact
metric space $\hat \Lambda$; here the homeomorphism $\hat f : \hat
\Lambda \to \hat \Lambda$ is the shift map defined by  $\hat f(x,
x_{-1}, x_{-2}, \ldots) = (f(x), x, x_{-1}, \ldots)$. So we obtain
for the non-invertible map $f$ and the equilibrium measure
$\mu_\phi$ the same estimate as above: $$A_\vp e^{S_n\phi(x) -
nP(\phi)} \le \mu_\phi(B_n(x, \vp)) \le B_\vp
e^{S_n\phi(x)-nP(\phi)},$$ with positive constants $A_\vp, B_\vp$
independent of $n, x$.

\begin{thm}\label{Jacobian}
Let $f$ be a smooth hyperbolic endomorphism on a folded basic set
$\Lambda$, which has no critical points in $\Lambda$; let also
$\phi$ a Holder continuous potential on $\Lambda$ and denote by
$\mu_\phi$ the unique equilibrium measure of $\phi$ on $\Lambda$.
Then for all $m \ge 1$, the Jacobian of $\mu_\phi$ w.r.t $f^m$ is
comparable to the ratio $\frac{\mathop{\sum}\limits_{\zeta\in
f^{-m}(f^m(x))\cap \Lambda} e^{S_m\phi(\zeta)}}{e^{S_m\phi(x)}}$,
i.e there exists a comparability constant $C >0$ (independent of
$m, x$) s.t for $\mu_\phi-\text{a.e} \ x \in \Lambda$:
\begin{equation}\label{jac}
C^{-1} \cdot \frac{\mathop{\sum}\limits_{\zeta\in f^{-m}(f^m(x)) \cap \Lambda}
e^{S_m\phi(\zeta)}}{e^{S_m\phi(x)}} \le J_{f^m}(\mu_\phi)(x) \le C
\cdot \frac{\mathop{\sum}\limits_{\zeta\in f^{-m}(f^m(x))\cap \Lambda}
e^{S_m\phi(\zeta)}}{e^{S_m\phi(x)}}, \
\end{equation}
\end{thm}

\begin{proof}
We know from definition that the Jacobian $J_{f^m}(\mu_\phi)$ is
the Radon-Nikodym derivative of $\mu_\phi\circ f^m$ with respect
to $\mu_\phi$ on sets of injectivity for $f^m$. In order to estimate the Jacobian of $\mu_\phi$
with respect to $f^m$, we have to compare the measure $\mu_\phi$
on different components of the preimage set $f^{-m}(B)$, for a
small borelian set $B$, where $m\ge 1$ is fixed. Let us consider
two subsets $E_1, E_2$ of $\Lambda$ so that $f^m(E_1) = f^m(E_2)
\subset B$ and $E_1, E_2$ belong to two disjoint balls $B_m(y_1,
\vp)$, respectively $B_m(y_2, \vp)$. This happens if the diameter
of $B$ is small enough, since $f$ has no critical points in
$\Lambda$ and thus there exists a positive distance $\vp_0$
between any two different preimages from $f^{-1}(y)$ for $y \in
\Lambda$.

As in \cite{KH}, since the borelian sets with boundaries of
measure zero form a sufficient collection, we can assume that each
of the sets $E_1, E_2$ have boundaries of $\mu_\phi$-measure zero.
We recall that $f^m(E_1) = f^m(E_2)$. But as in \cite{KH},
$\mu_\phi$ is the limit of the sequence of measures: $$\tilde
\mu_n:= \frac {1}{P(f, \phi, n)} \cdot \mathop{\sum}\limits_{x \in
\text{Fix}(f^n) \cap \Lambda} e^{S_n\phi(x)} \delta_x,$$ where $P(f, \phi, n):=
\mathop{\sum}\limits_{x \in \text{Fix}(f^n)\cap \Lambda} e^{S_n\phi(x)}, n \ge
1$. So we obtain
\begin{equation}\label{ai}
\tilde \mu_n(E_1) = \frac {1}{P(f, \phi, n)} \cdot
\mathop{\sum}\limits_{x \in \text{Fix}(f^n) \cap E_1}
e^{S_n\phi(x)}, n \ge 1
\end{equation}

Let us now consider a periodic point $x \in \text{Fix}(f^n) \cap
E_1$; it follows that $f^m(x) \in
f^m(E_1)$, so there exists a point $y \in E_2$ such that $f^m(y) =
f^m(x)$. However the point $y$ is not necessarily periodic. Hence we will use the Specification Property
(\cite{KH}, \cite{Bo}) on hyperbolic locally maximal sets
in order to approximate $y$ with a periodic point whose orbit
follows that of $y$ for sufficiently long time. Indeed if $\vp>0$
is fixed, there exists a constant $M_\vp>0$ such that for all
$n > M_\vp$, there is a point $z \in \text{Fix}(f^n) \cap \Lambda$ which
$\vp$-shadows the $(n-M_\vp)$-orbit of $y$. In particular $z \in B_m(y_2,
2\vp)$, since $E_2 \subset B_m(y_2, \vp)$.

Let now $V\subset B_m(y_2, \vp)$ be an arbitrary neighbourhood of the set $E_2$.  Let
us take two points $x, x' \in \text{Fix}(f^n)\cap E_1$ and
assume the same periodic point $z \in V \cap \text{Fix} (f^n)$
corresponds to both of them through the previous shadowing procedure. Thus
the $(n- M_\vp-m)$-orbit of $f^m(z)$ $\vp$-shadows the
$(n-M_\vp-m)$-orbit of $f^m(x)$ and also the $(n-M_\vp-m)$-orbit
of $f^m(x')$. Thus the $(n-M_\vp-m)$-orbit of $f^m(x)$
$2\vp$-shadows the $(n-M_\vp-m)$-orbit of $f^m(x')$. But
recall that we took $x,  x' \in E_1 \subset B_m(y_1, \vp)$, so $x' \in B_m(x, 2\vp)$ and  hence from above, $x' \in B_{n-M_\vp}(x, 2\vp)$. \
We will partition now the set $B_{n-M_\vp}(x, 2\vp)$ in at most
$N_\vp$ smaller Bowen balls of type $B_n(\zeta, 2\vp)$. In each of
these $(n, 2\vp)$-Bowen balls we may have at most one fixed point
for $f^n$. Indeed, fixed points for $f^n$ are solutions to the
equation $f^n \xi = \xi$ and $Df^n$ does not have unitary
eigenvalues. Then if $d(f^i \xi, f^i \zeta) < 2\vp, i = 0,
\ldots, n-1$ and if $\vp$ is small enough, we can apply the
Inverse Function Theorem at each step, and thus there exists
only one fixed point for $f^n$ in the Bowen ball $B_n(\zeta,
2\vp)$. \ So there may exist at most $N_\vp$ periodic points in $\Lambda$ from
$\text{Fix}(f^n) \cap E_1$ having the same point $z \in V \cap \text{Fix}(f^n)$
associated to them by the above shadowing correspondence.

Let us notice also that if $x,  x' \in \text{Fix}(f^n) \cap E_1$ have the
same point $z\in V$ attached to them, then as seen before, $x' \in B_{n-M_\vp}(x, 2\vp)$ and then, from the Holder continuity
of $\phi$, $$|S_n\phi(x) - S_n\phi(x')| \le \tilde C_\vp,$$
for some positive constant $\tilde C_\vp$ depending on $\phi$ (but
independent of $n, m, x$). This can be used then in the estimate for
$\tilde \mu_n(E_1)$, from (\ref{ai}). Notice also that, if
$z \in B_{n-M_\vp}(y, \vp)$, then $f^m(z) \in
B_{n-M_\vp-m}(f^m(x), \vp)$. Thus from the Holder continuity of $\phi$ and the fact that $x \in E_1
\subset B_m(y_1, \vp)$, it follows that there exists a positive constant
$\tilde C_\vp'$ satisfying:
\begin{equation}\label{Sn}
|S_n \phi(z) - S_n\phi(x)| \le |S_m\phi(y_1) - S_m\phi(y_2)| +
\tilde C_\vp', \ \text{for} \ n > n(\vp, m).
\end{equation}

Then from (\ref{Sn}), (\ref{ai}), and since there are at most
$N_\vp$ points $x \in \text{Fix}(f^n) \cap E_1$ having the same $z \in V
\cap \text{Fix}(f^n) \cap \Lambda$ corresponding to them, we obtain that there
exists a constant $C_\vp>0$ s.t:
\begin{equation}\label{ineq1}
\tilde \mu_n(E_1) \le C_\vp \tilde \mu_n(V) \cdot
\frac{e^{S_m\phi(y_1)}}{e^{S_m\phi(y_2)}},
\end{equation}
where we recall that $E_1 \subset B_m(y_1, \vp), E_2 \subset B_m(y_2, \vp)$ and $f^m(E_1) = f^m(E_2)$. But $\partial E_1,
\partial E_2$ were assumed of $\mu_\phi$-measure zero, hence: $$
\mu_\phi(E_1) \le C_\vp \mu_\phi(V)
\cdot \frac{e^{S_m\phi(y_1)}}{e^{S_m\phi(y_2)}}$$

Recall now that $V$ was chosen arbitrarily as a neighbourhood of
$E_2$, and by applying the same procedure for $E_1$ instead
of $E_2$ we obtain the estimates:
\begin{equation}\label{comparison}
\frac{1}{C} \mu_\phi(E_2)
\frac{e^{S_m\phi(y_1)}}{e^{S_m\phi(y_2)}} \le \mu_\phi(E_1)  \le
C \mu_\phi(E_2) \frac{e^{S_m\phi(y_1)}}{e^{S_m\phi(y_2)}},
\end{equation}
where $C>0$ does not depend on $m, E_1, E_2$.

Now the Jacobian $J_{f^m}(\mu_\phi)$ is the Radon-Nikodym
derivative of $\mu_\phi \circ f^m$ with respect to $\mu_\phi$ on
sets of injectivity for $f^m$, hence $$\mu_\phi(f^m(D)) = \int_D
J_{f^m}(\mu_\phi)(x)d\mu_\phi(x), $$ for any borelian set $D$ on
which $f^m$ is injective. And on the other hand from the
invariance of $\mu_\phi$, we have $\mu_\phi(f^m(D)) =
\mu_\phi(f^{-m}(f^m D))$. Thus from (\ref{comparison}), the fact
that $|S_m\phi(\zeta) - S_m\phi(y)| \le \tilde C_\vp$ for $\zeta
\in B_m(y, \vp)$ and from the Lebesgue Derivation Theorem, it
follows that the Jacobian of $\mu_\phi$ satisfies:
$$J_{f^m}(\mu_\phi)(x) \approx
\frac{\mathop{\sum}\limits_{\zeta\in f^{-m}(f^m(x))\cap \Lambda}
e^{S_m\phi(\zeta)}}{e^{S^m\phi(x)}}, \ \mu_\phi-\text{a.e} \ x \in
\Lambda,$$ where the comparability constant $C>0$ is independent
of $m>1, x \in \Lambda$.

\end{proof}

Let us give now the definition of the folding entropy and the entropy production according to Ruelle, \cite{Ru-folding}.

\begin{defn}\label{fe}
Let $f:M \to M$ be a smooth endomorphism and
$\mu$ an $f$-invariant probability on $M$, then the \textbf{folding
entropy} $F_f(\mu)$ of $\mu$ is the conditional entropy: $$F_f(\mu) :=
H_\mu(\epsilon|f^{-1}\epsilon),$$ where $\epsilon$ is the
partition into single points. Also define the \textbf{entropy
production} of $\mu$ by: $$e_f(\mu):= F_f(\mu) - \int \log
|\text{det} Df(x)| d\mu(x)$$
\end{defn}

From \cite{Ro} it follows that we can use the measurable single
point partition $\epsilon$ in order to desintegrate the invariant
measure $\mu$ into a canonical family of conditional measures
$\mu_x$ supported on the finite fiber $f^{-1}(x)$ for $\mu$-a.e
$x$. Thus the entropy of the conditional measure of $\mu$ restricted to
$f^{-1}(x)$ is $H(\mu_x) = -\Sigma_{y \in f^{-1}(x)} \mu_x(y) \log
\mu_x(y)$. From \cite{Pa} we have also $$J_f(\mu)(x) =
\frac{1}{\mu_{f(x)}(x)}, \ \mu-\text{a.e} \ x,$$  hence we obtain
that
\begin{equation}\label{fold}
F_f(\mu) = \int \log J_f(\mu)(x) d \mu(x)
\end{equation}

Let us return now to the case of a hyperbolic basic set $\Lambda$ for a smooth endomorphism $f$ and consider a Holder
potential $\phi$ on $\Lambda$, with its unique equilibrium measure $\mu_\phi$.
We will give a formula for the folding entropy of the equilibrium measure $\mu_\phi$ in terms of an
"asymptotic logarithmic degree" with respect to $\mu_\phi$. This will take into account
the $n$-preimages of points which behave well (are generic) with respect to $\mu_\phi$.
To this end, for an $f$-invariant probability (borelian) measure $\mu$ on $\Lambda$ let us define, for any small
$\tau>0$,  $n >0$ integer and $x \in \Lambda$ the set

\begin{equation}\label{G}
G_n(x, \mu, \tau):= \{y \in f^{-n}(f^nx) \cap \Lambda, \ \text{s.t} \
|\frac{S_n\phi(y)}{n} - \int \phi d\mu| < \tau\},
\end{equation}
where $S_n\phi(y) := \phi(y) + \ldots + \phi(f^{n-1}y), y \in
\Lambda$ is the consecutive sum of $\phi$ on $y$.

\begin{defn}\label{dn}
In the above setting, denote by $d_n(x, \mu, \tau):= \text{Card} G_n(x, \mu, \tau), x \in \Lambda, n >0, \tau >0$.
The function $d_n(\cdot, \mu, \tau)$ is measurable, nonnegative and finite on $\Lambda$.
\end{defn}

\begin{thm}\label{folding-ent}
Let $f:M \to M$ be a smooth endomorphism and $\Lambda$ a basic set
for $f$ so that $f$ is hyperbolic on $\Lambda$ and does not have
critical points in $\Lambda$. Let also $\phi$ a Holder continuous
potential on $\Lambda$ and $\mu_\phi$ the equilibrium measure
associated to $\phi$. Then we have the following formula for the
folding entropy of $\mu_\phi$: $$F_f(\mu_\phi) =
\mathop{\lim}\limits_{\tau \to 0} \mathop{\lim}\limits_{n \to
\infty} \frac 1n \int_\Lambda \log d_n(x, \mu_\phi, \tau)
d\mu_\phi(x)$$
\end{thm}

\begin{proof}
First let us recall formula (\ref{fold}) for an arbitrary
$f$-invariant measure $\mu$, namely $$F_f(\mu) = \int_\Lambda\log
J_f(\mu)(x) d\mu_(x)$$ From the Chain Rule for Jacobians,
$J_{f^n}(\mu)(x) = J_f(\mu)(x) \ldots J_f(\mu)(f^{n-1}(x))$
$\mu$-a.e, for any $n \ge 1$. On the other hand, since $\mu$ is
$f$-invariant, we have that $$\int \log J_f(\mu)(x) d\mu(x) = \int
\log J_f(\mu)(f(x)) d\mu(x) = \int \log J_f(\mu)(f^k x)d\mu(x),$$
for all $k \ge 1$. These facts imply that for any $n \ge 1$,
\begin{equation}\label{n}
F_f(\mu) = \frac 1n \int \log J_{f^n} (\mu)(x) d\mu(x)
\end{equation}

Therefore from Theorem \ref{Jacobian}, since the constant $C$ is
independent of $n$ we obtain that:
\begin{equation}\label{sigma}
F_f(\mu_\phi) = \mathop{\lim}\limits_{n \to \infty} \frac 1n
\int_\Lambda \log \frac{\mathop{\sum}\limits_{y \in
 f^{-n}(f^n(x))\cap \Lambda} e^{S_n\phi(y)}}{e^{S_n\phi(x)}} d\mu_\phi(x)
\end{equation}

Now since $\Lambda$ is compact, each point $x \in \Lambda$ has only finitely many $f$-preimages in $\Lambda$,
i.e there exists a positive integer $d$ s.t $\text{Card}(f^{-1}x) \le d, x \in \Lambda$.

Since $\mu_\phi$ is an ergodic measure and from Birkhoff Ergodic
Theorem we obtain that
 $\mu_\phi(x \in \Lambda, |\frac{S_n\phi(x)}{n} - \int \phi d\mu| >
 \tau/2)
 \mathop{\to}\limits_{n \to \infty} 0$, for any small $\tau >0$. Thus for any $\eta >0$ there exists a large integer
 $n(\eta)$ s.t for $n \ge n(\eta)$,
\begin{equation}\label{birkhoff}
\mu_\phi(x \in \Lambda, |\frac{S_n\phi(x)}{n} - \int \phi d\mu| >
\tau/2) < \eta
\end{equation}

Let us now take a point $x \in \Lambda$ with $|\frac{S_n
\phi(x)}{n} - \int \phi d\mu| < \tau$. From Definition \ref{dn} we
have
\begin{equation}\label{suma}
\frac{e^{n(\int \phi d\mu_\phi - \tau)} d_n(x, \mu_\phi, \tau) +
r_n(x, \mu_\phi, \tau)}{e^{n(\int \phi d\mu + \tau)}} \le
\frac{\mathop{\sum}\limits_{y \in f^{-n}(f^n x) \cap \Lambda}
e^{S_n\phi(y)}}{e^{S_n\phi(x)}} \le \frac{e^{n(\int \phi d\mu_\phi
+ \tau)} d_n(x, \mu_\phi, \tau) + r_n(x, \mu_\phi,
\tau)}{e^{n(\int \phi d\mu_\phi - \tau)}},
\end{equation}

where $r_n(x, \mu_\phi, \tau)$ is the remainder
$\mathop{\sum}\limits_{y \in f^{-n}f^n(x) \setminus G_n(x,
\mu_\phi, \tau)} e^{S_n\phi(y)}$. \ In order to simplify notation,
we will also denote $r_n(x, \mu_\phi, \tau)$ by $r_n$ when no
confusion can arise.

Given $n$ large, let us consider now a partition $(A_i^n)_{1 \le i
\le K}$ of $\Lambda$ (modulo $\mu_\phi$) so that for
 each $0 \le i \le K$, there exists a point $z_i \in A_i^n$ so that for any $n$-preimage $\xi_{ij} \in f^{-n}(z_i)
 \cap \Lambda, 1 \le j \le d_{n, i}$, we have $A_i^n \subset f^n(B_n(\xi_{ij}, \vp)), 1 \le j \le d_{n, i},  1 \le i \le K$.
For the above partition, let us denote by $A^n_{ij}$ the part
of the $n$-preimage of $A^n_i$ which belongs to the Bowen ball
$B_n(\xi_{ij}, \vp)$, i.e $A^n_{ij}:= f^{-n}(A^n_i) \cap
B_n(\xi_{ij}, \vp), 1 \le j \le d_{n, i}, 1 \le i \le K$. Since
the sets $A^n_i$ were chosen disjoint, also the pieces of their
preimages, namely $A^n_{ij}, i, j$, are mutually disjoint.

We will decompose the integral in (\ref{sigma}) over the sets
$A^n_{ij}$. Notice that if $y, z \in A^n_{ij}$, then since $\phi$
is Holder continuous and $A_{ij}^n \subset B_n(\xi_{ij}, \vp)$, it
follows that we have
\begin{equation}\label{sn-phi}
|S_n \phi(y) - S_n\phi(z)| \le C(\vp),
\end{equation}
 where $C(\vp)$ is a
positive function with $C(\vp) \mathop{\to}\limits_{\vp \to 0} 0$.
So we will obtain now:
\begin{equation}\label{inte}
\int_\Lambda \log \frac{\mathop{\sum}\limits_{y \in f^{-n}f^n x
\cap \Lambda} e^{S_n\phi(y)}}{e^{S_n\phi(x)}} d\mu_\phi(x) =
\mathop{\sum}\limits_{0 \le j \le d_i, 0 \le i \le K}
\int_{A^n_{ij}} \log \frac{\mathop{\sum}\limits_{y \in f^{-n}f^n x
\cap \Lambda }e^{S_n\phi(y)}}{e^{S_n\phi(x)}} d\mu_\phi(x)
\end{equation}

Let us now denote by $R_n(i, \mu_\phi, \tau)$ the set of preimages
$\xi_{ij}$ with $\xi_{ij} \notin G_n(\xi_{ik_0}, \mu_\phi, \tau)$,
and denote simply by $R_{n, i}$ the set of indices $j, 1 \le j \le
d_{n, i} $ with $\xi_{ij} \in R_n(i, \mu_\phi, \tau)$ for every $1
\le i \le K$. Now in the decomposition from (\ref{inte}) we notice
that the integral over those sets $A^n_{ij}$ with $j \in R_{n, i}$
will not matter significantly. Indeed as $\text{Card} (f^{-1}x
\cap \Lambda) \le d, x \in \Lambda$ and since $-M \le \phi(x) \le
M, x \in \Lambda$ we have $$1 \le \frac{\mathop{\sum}\limits_{y
\in f^{-n}f^n x \cap \Lambda}e^{S_n\phi(y)}}{e^{S_n\phi(x)}} \le
d^n e^{2nM}$$

Now recall that each $A^n_{ij} \subset B_n(\xi_{ij}, \vp)$ and the
sets $A^n_{ij}, i, j$ are mutually disjoint (with respect to
$\mu_\phi$). Hence by using inequalities (\ref{birkhoff}) and
(\ref{sn-phi}) and the fact that $\xi_{ij} \notin G_n(\xi_{ik_0},
\mu_\phi, \tau)$ whenever $j \in R_{n, i}$, we obtain:
\begin{equation}\label{neglij}
\mathop{\sum}\limits_{0 \le i \le K, j \in R_{n, i}} \frac 1n
\int_{A^n_{ij}} \log \frac{\mathop{\sum}\limits_{y \in f^{-n}f^n x
\cap \Lambda }e^{S_n\phi(y)}}{e^{S_n\phi(x)}} d\mu_\phi(x) \le
\frac 1n \log (d^n e^{2nM}) \cdot \eta = \eta (\log d + 2M)
\end{equation}

But by using the comparison between different parts of the
$n$-preimage of a small set from the
proof of Theorem \ref{Jacobian} (see (\ref{comparison})), we deduce that the last term of
formula (\ref{inte}) is comparable to
\begin{equation}\label{logi}
 \mathop{\sum}\limits_{i, j}
\mu_\phi(A^n_{ij}) \log \frac{d_n(z_i, \mu_\phi, \tau)
\mu_\phi(A^n_{ij}) + \tilde r_n(z_i, \mu_\phi,
\tau)}{\mu_\phi(A^n_{ij})},
\end{equation}
where $\tilde r_n(z_{i}, \mu, \tau) :=
\mathop{\sum}\limits_{\xi_{ij}\in f^{-n}(z_i)\cap \Lambda, \
\xi_{ij} \notin G_n(\xi_{ik_0}, \mu_\phi, \tau)} \mu_\phi(A^n_{ij})$,

Hence from (\ref{comparison}), (\ref{neglij}) and (\ref{logi}) we obtain:
\begin{equation}\label{m}
\aligned &\frac 1n\mathop{\sum}\limits_{i, j \notin R_{n, i}}
\mu_\phi(A^n_{ij}) \log d_n(z_i, \mu_\phi, \tau) + \frac 1n
\mathop{\sum}\limits_{i, j \notin R_{n, i}} \mu_\phi(A^n_{ij})
\log (1 + \frac{\tilde r_n(z_i, \mu_\phi, \tau)}{d_n(z_i,
\mu_\phi, \tau)\mu_\phi(A^n_{ij})}) - \delta(\tau) -\eta C' \le \\
& \le \int_\Lambda \frac 1n \log \frac{\mathop{\sum}\limits_{y \in
f^{-n}f^n x \cap \Lambda} e^{S_n\phi(y)}}{e^{S_n\phi(x)}} d\mu_\phi(x)  \le
\\ & \le \frac 1n \mathop{\sum}\limits_{i, j \notin R_{n, i}} \mu_\phi(A^n_{ij}) \log d_n(z_i,
\mu_\phi, \tau) + \frac 1n \mathop{\sum}\limits_{i, j \notin R_{n,
i}} \mu_\phi(A^n_{ij}) \log (1 + \frac{\tilde r_n(z_i, \mu_\phi,
\tau)}{d_n(z_i, \mu_\phi, \tau)\mu_\phi(A^n_{ij})}) + \delta(\tau)
+ \eta C',
\endaligned
\end{equation}
with $C' = \log d + 2M$ being the constant found in (\ref{neglij}), and
where the positive constant $\delta(\tau)$ comes from the
uniformly bounded variation of $\frac 1n S_n\phi(x)$ when $x$ is in
$A^n_{ij}$ and when $1\le i \le K, j \notin R_{n, i}$ vary; clearly we have $\delta(\tau)
\mathop{\to}\limits_{\tau \to 0} 0$.

Now we know that in general
$\log(1+x) \le x$, for $x
>0$. Thus $ \log (1 + \frac{\tilde r_n(z_i, \mu_\phi,
\tau)}{d_n(z_i, \mu_\phi, \tau)\mu_\phi(A^n_{ij})}) \le
\frac{\tilde r_n(z_i, \mu_\phi, \tau)}{d_n(z_i, \mu_\phi,
\tau)\mu_\phi(A^n_{ij})}, i, j$ and hence in (\ref{m}) we have, for $n$ large enough that:
\begin{equation}\label{rn}
\aligned &\mathop{\sum}\limits_{i, j \notin R_{n, i}}
\mu_\phi(A^n_{ij}) \log (1 + \frac{\tilde r_n(z_i, \mu_\phi,
\tau)}{d_n(z_i, \mu_\phi, \tau)\mu_\phi(A^n_{ij})}) \le
\mathop{\sum}\limits_{i, j \notin R_{n, i}} \mu_\phi(A^n_{ij})
\frac{\tilde r_n(z_i, \mu_\phi, \tau)}{d_n(z_i, \mu_\phi,
\tau)\mu_\phi(A^n_{ij})} = \\ &= \mathop{\sum}\limits_{1 \le i \le
K} \tilde r_n(z_i, \mu_\phi, \tau) \le \eta,
\endaligned
\end{equation}
 where we used that by definition, there are $d_n(z_i, \mu_\phi, \tau)$ indices $j$ in
$\{1, \ldots, d_{n, i}\} \setminus R_{n, i}$ for any $1 \le i \le
K$. Therefore from the last displayed inequality and from
(\ref{m}) we obtain, for $n \ge n(\eta)$, that:
\begin{equation}\label{f}
\left| \frac 1n \int_\Lambda \log \frac{\mathop{\sum}\limits_{y
\in f^{-n}f^n x \cap \Lambda} e^{S_n\phi(y)}}{e^{S_n\phi(x)}}
d\mu_\phi(x) - \frac 1n \int_\Lambda \log d_n(z, \mu_\phi, \tau)
d\mu_\phi(z) \right| \le \delta(\tau) + \eta,
\end{equation}
where $\delta(\tau) \mathop{\to}\limits_{\tau \to 0} 0$. Then by
taking $n \to \infty$ and $\tau \to 0$, we will obtain the
conclusion of the Theorem from (\ref{sigma}) and (\ref{f}), namely
that $$F_f(\mu_\phi) = \mathop{\lim}\limits_{\tau \to 0}
\mathop{\lim}\limits_{n \to \infty} \frac 1n \int_\Lambda \log
d_n(x, \mu_\phi, \tau) d\mu_\phi(x)$$

\end{proof}

\begin{cor}\label{neg-ent}
 a) Let $f: \mathbb T^m \to \mathbb T^m, m \ge 2$ be a hyperbolic toral endomorphism, and $\phi$ be an arbitrary Holder continuous
potential on $\mathbb T^m$, with its associated equilibrium
measure $\mu_\phi$. Then the entropy production of $\mu_\phi$ is
non-positive, i.e $$e_f(\mu_\phi) \le 0$$ In the same setting the
entropy production of the Haar (Lebesgue) measure is equal to 0.

b) The same conclusions as above hold also for any Anosov endomorphism $f: \mathbb T^m \to \mathbb T^m$ with constant Jacobian with respect to the Riemannian metric, i.e for which $\text{det} Df$ is constant on $\mathbb T^m$.
\end{cor}

\begin {proof}

a) In the case of a toral endomorphism $f$ given by the
integer-valued matrix $A$, the determinant of the derivative
$\text{det} Df$ is constant and equal to $\text{det} A$. Thus
$$\int_{\mathbb T^m} \log |\text{det} Df| d\mu_\phi = \log d,$$
where $d := |\text{det} A|$. On the other hand, by looking at the
area of $f(I \times \ldots \times I)$, it is easy to see that $d$
is exactly the number of $f$-preimages that any point from
$\mathbb T^m = I \times \ldots \times I$ ($m$ times) has.
Therefore, by taking $\Lambda = \mathbb T^m$ and by recalling
Definition \ref{dn}, one obtains that $$d_n(x, \mu_\phi, \tau) \le
d^n, \ \forall x \in \mathbb T^m, n >0, \tau>0$$ Hence from Theorem
\ref{folding-ent} it follows that $$e_f(\mu_\phi) \le 0$$ For the
last statement of a), we have that $f$ invariates the Lebesgue
measure $m$, that $|\text{det} Df|$ is constant and equal to $d$
and that $d_n(x, m, \tau)$ is constant in $x$ and equal to $d$
since the Lebesgue (Haar) measure is the unique measure of maximal
entropy. Therefore the entropy production of the Lebesgue measure
$m$ with respect to $f$ is equal to 0.

The last statement of a) can also be obtained from the fact that
the entropy production of invariant absolutely continuous measures
is non-negative (from \cite{Ru-folding}), combined with the first
part of the proof.

b) The argument is the same as for a), namely if $\text{det}Df$ is
constant, then $f$ invariates the Lebesgue measure $m$, and it is
$d$-to-1, for $d = |\text{det} Df|$. Then $d_n(x, \mu_\phi, \tau)
\le d$ for any $x, \tau, n$ and $e_f(\mu_\phi) \le 0$.

\end{proof}

However we will see later that Corollary \ref{neg-ent} is no
longer true for perturbations of a toral endomorphism $f$, and
that there exist equilibrium measures of Holder potentials which
have in certain cases positive entropy production.

Theorem \ref{folding-ent} also helps us calculate the folding
entropy of the measure of maximal entropy for a general hyperbolic
(hence non-expanding) endomorphism. Then by knowing this, one can
calculate the \textbf{entropy production of the measure of maximal
entropy}, from Definition \ref{fe}.

\begin{cor}\label{deg-ent}
In the setting of Theorem \ref{folding-ent}, denote by $\mu_0$ the
unique measure of maximal entropy for $f$ on $\Lambda$. If
$d_n(x)$ denotes the cardinality of $f^{-n}(f^nx) \cap \Lambda$
for $n \ge 1$, then we have: $$F_f(\mu_0) =
\mathop{\lim}\limits_{n \to \infty} \frac 1n \int_\Lambda \log
d_n(x) d\mu_0(x)$$ In particular if $f$ is $d$-to-1 on $\Lambda$,
then $F_f(\mu_0) = \log d$.
\end{cor}

 Let us now recall the notion of
\textbf{inverse SRB measure}, introduced in \cite{M-JSP}. These
measures exist in the case of hyperbolic repellers (and in
particular in the case of Anosov endomorphisms) and are physically
relevant
 since they describe the \textit{past trajectories} of Lebesgue almost all points in a neighbourhood of the repellor.
 We will show that there exist Anosov endomorphisms, whose respective inverse SRB measures
 have
 \textbf{negative entropy production}.

Let $\Lambda$ be a connected hyperbolic repeller for a smooth
endomorphism $f:M \to M$ defined on a Riemannian manifold $M$, and
assume $f$ has no critical points in $\Lambda$. Let $V$ be a
neighbourhood of $\Lambda$ in $M$ and for any $z \in V$ define the
measures
\begin{equation}\label{mun}
\mu_n^z := \frac 1n \mathop{\sum}\limits_{y \in f^{-n}z \cap V}
\frac {1}{d(f(y)) \ldots d(f^n(y))} \mathop{\sum}\limits_{i=1}^n
\delta_{f^i y},
\end{equation}
where $d(y)$ is the number of $f$-preimages belonging to $V$ of a
point $y \in V$ ($d(\cdot)$ is called also the degree function).

Then we proved in \cite{M-JSP} that there exists an $f$-invariant
measure $\mu^-$ on $\Lambda$, a neighbourhood $V$ of $\Lambda$ and
a borelian set $A \subset V$ with $m(V \setminus A) = 0$ (where
$m$ is the Lebesgue measure on $M$) and a subsequence $n_k \to
\infty$ such that for any $z \in A$,
\begin{equation}\label{converg}
\mu^z_{n_k} \mathop{\to}\limits_{k \to \infty} \mu^-
\end{equation}
 The measure
$\mu^-$ is called the \textbf{inverse SRB measure} of the
hyperbolic repeller. We showed in \cite{M-JSP}  that $\mu^-$ is
the equilibrium measure of the stable potential $\Phi^s(x):= \log
|\text{det} Df_s(x)|, x \in \Lambda$, with respect to $f$ (where
we recall the notation from (\ref{DSU})). The difficulty is that
the map $f$ is non-invertible, hence $\mu^-$ is \textbf{not}
simply the SRB measure for the inverse $f^{-1}$. Moreover from the
hyperbolicity condition in the case of endomorphisms, the unstable
manifolds may intersect each other both in $\Lambda$ and outside
$\Lambda$ and through any point of $\Lambda$ there may pass
infinitely many (even uncountably many, as shown in \cite{M-MZ})
unstable manifolds.

We also proved that this inverse SRB measure $\mu^-$ is the unique $f$-invariant measure $\mu$ satisfying an \textit{inverse Pesin entropy formula:} in the case when $f$ is $d$-to-1 on $\Lambda$
\begin{equation}\label{invP}
h_\mu(f) = \log d - \int_\Lambda \mathop{\sum}\limits_{i, \lambda_i(\mu, x) <0} \lambda_i(\mu, x) m_i(\mu, x) d\mu(x),
\end{equation}
where $\lambda_i(\mu, x)$ are the Lyapunov exponents of the
measure $\mu$ at $x$ and $m_i(\mu, x)$ are the respective
multiplicities of these Lyapunov exponents. In addition if $f$ is
$d$-to-1 on the connected hyperbolic repeller $\Lambda$, then the
inverse SRB measure $\mu^-$ has \textit{absolutely continuous
conditional measures on local stable manifolds} (see
\cite{M-JSP}).

Also for an Anosov endomorphism $f$ on $M$, we know from
\cite{QZ}, \cite{QS} that there exists a unique SRB measure
$\mu^+$ which satisfies a Pesin entropy formula and which is the
projection $\pi_*$ of the equilibrium measure of the unstable
potential $\Phi^u(\hat x):= -\log |\text{det} Df_u(\hat x)|, \hat
x \in \hat M$ (with the notation for the unstable derivative from
(\ref{DSU})).

We prove now that the entropy production of the respective inverse SRB
measure of a perturbation $g$ of a hyperbolic toral endomorphism, is less than
or equal to 0; we identify the cases when it is 0 as exactly
those cases when $\mu_g^-$ is absolutely continuous on $\mathbb
T^m$.

\begin{thm}\label{ent-tor}
Let $f$ be a hyperbolic toral endomorphism on $\mathbb T^m, m \ge
2$ given by an integer-valued matrix $A$ without zero eigenvalues,
and let $g$ be a $\mathcal{C}^1$ perturbation of $f$. Consider
$\mu_g^-$ the inverse SRB measure of $g$ and $\mu_g^+$ the (usual
forward) SRB measure. Then:

a)  $e_g(\mu^-_g) \le 0$ and $F_g(\mu_g^-) = \log d$. Moreover
$e_g(\mu^+_g) \ge 0$.

b) $e_g(\mu^-_g) = 0$ if and only if $|\text{det} Dg|$ is
cohomologous to a constant on $\mathbb T^m$. Same condition on
$|\text{det} Dg|$  holds if and only if $e_g(\mu_g^+) = 0$. In
either case we obtain $\mu_g^- = \mu_g^+$, and the common value is
absolutely continuous with respect to the Lebesgue measure on
$\mathbb T^m$.
\end{thm}

\begin{proof}

a) If $f$ is given by an integer valued matrix $A$, then $f$ is
$d$-to-1 on $\mathbb T^m$, where $d= |\text{det} A|$. If $g$ is a
$\mathcal{C}^1$ perturbation of the hyperbolic toral endomorphism
$f$, then it is clear that
 $g$ is also hyperbolic on $\mathbb T^m$. Thus from \cite{QZ} we can construct the SRB measure of $g$, denoted by
 $\mu_g^+$, which is the projection by $\pi_*$ of the equilibrium measure of
 $\Phi^u_g(\hat x) = - \log |\text{det} Dg_u(\hat x)|, \hat x  \in  \hat{\mathbb T^m}$. In particular $\mu_g^+$ is ergodic,
 hence its Lyapunov exponents are constant $\mu_g^+$-a.e.

From the discussion above, since $f$ has no critical points, we
can construct the inverse SRB measure $\mu_g^-$ which is the
equilibrium measure of the stable potential $\Phi^s_g(x) = \log
|\text{det}Dg_s(x)|, x \in \mathbb T^m$; thus $\mu_g^-$ is ergodic
too, and its Lyapunov exponents are constant $\mu_g^-$-a.e on
$\mathbb T^m$.

Now since $g$ is a perturbation of $f$, it follows that every
point in $\mathbb T^m$ has exactly $d$ $g$-preimages, where $d =
|\text{det} A|$. Thus from \cite{M-JSP}, it follows that $\mu^-_g$
is the weak limit of a sequence of measures of type (\ref{mun}),
where the degree function $d(\cdot)$ is constant and equal to $d$
everywhere on $\mathbb T^m$. This implies then that the Jacobian
of $\mu_g^-$ is constant and equal to $d$, since for any small
borelian set $B$, we have that a point $x \in g(B)$ if and only if
there is exactly one $g$-preimage $x_{-1}$ of $x$ in $B$, and we
use this fact in the above convergence (\ref{converg}) of measures
towards $\mu^-$.  Hence $$F_g(\mu_g^-) = \int \log
J_{g}(\mu_g^-)(x) d\mu_g^-(x) = \log d$$ And from (\ref{invP}) we
have that $$ h_{\mu_g^-}(g) = \log d -
\mathop{\sum}\limits_{\lambda_i(\mu_g^-) <0} \lambda_i(\mu_g^-) $$
Thus if $e_g(\mu_g^-) >0$,  it would follow that $F_g(\mu_g^-)
> \int \log |\text{det} Dg| d\mu_g^- = \frac 1n \int \log
|\text{det} Dg^n| d\mu_g^-, n \ge 1$. Hence from the last
displayed formula and Birkhoff Ergodic Theorem, we obtain
$h_{\mu_g^-}(g)
> \mathop{\sum}\limits_{\lambda_i(\mu_g^-) > 0}
\lambda_i(\mu_g^-)$, which gives a contradiction with Ruelle's
inequality. Therefore we have for any perturbation $g$, $$
e_g(\mu_g^-) \le 0 $$ Now for the SRB measure $\mu_g^+$: if the
entropy production $e_g(\mu_g^+)$ were strictly negative, then
$F_g(\mu_g^+) < \int \log|\text{det} Dg| d\mu_g^+$.  Since from
\cite{PDL}, $h_{\mu_g^+}(g) \le F_g(\mu^+_g) -
\mathop{\sum}\limits_{\lambda_i(\mu_g^+) <0} \lambda_i(\mu^+_g)$,
it would follow that $h_{\mu_g^+}(g) <
\mathop{\sum}\limits_{\lambda_i(\mu_g^+) >0} \lambda_i(\mu_g^+)$,
which is a contradiction to the fact that the SRB measure
satisfies Pesin entropy formula. Consequently, $$ e_g(\mu_g^+) \ge
0$$

\

b) If $e_g(\mu_g^-) = 0$, then $F_g(\mu_g^-) = \int \log
|\text{det} Dg| d\mu_g^-$; hence from the Birkhoff Ergodic Theorem
and \cite{PDL} we obtain: $$h_{\mu_g^-}(g) = \int \log |\text{det}
Dg| d\mu_g^- - \mathop{\sum}\limits_{\lambda_i(\mu_g^-)<0}
\lambda_i(\mu_g^-) = \mathop{\sum}\limits_{\lambda_i(\mu_g^- >0}
\lambda_i(\mu_g^-)$$ Therefore from the uniqueness of the
$g$-invariant measure satisfying Pesin entropy formula, we obtain
that $\mu_g^- = \mu_g^+$.

Recalling from above that $\mu_g^-$ is the equilibrium measure of
the stable potential $\Phi^s$ and $\mu_g^+$ is the equilibrium
measure of the unstable potential $\Phi^u$, we see from Livshitz
Theorem (see \cite{KH}), that $\mu_g^- = \mu_g^+$ if and only if
$\text{det} Dg$ is cohomologous to a constant.

Assume now that $\mu_g^+ = \mu_g^-$; then since $\mu_g^+$ has
absolutely continuous conditional measures associated to a
partition subordinated to local unstable manifolds (\cite{QZ},
\cite{QS}) and $\mu_g^-$ has absolutely continuous conditional
measures associated to a partition subordinated to local stable
manifolds (from \cite{M-JSP}), we obtain that $\mu_g^+$ is absolutely
continuous with respect to the Lebesgue measure on $\mathbb T^m$.

\end{proof}

\begin{cor}\label{ruelle}
In the setting of Theorem \ref{ent-tor}, let $g$ be a perturbation of the hyperbolic toral endomorphism $f$ s.t
 $|\text{det} Dg|$ is not cohomologous to a constant. Then its unique inverse SRB measure $\mu_g^-$ is not a weak
 limit of a sequence of type (\ref{rho-n}).
 \end{cor}

\begin{proof}
As was proved in \cite{Ru-folding}, the entropy production of any limit of measures of type (\ref{rho-n}) is nonnegative.
On the other hand, if $|\text{det}Dg|$ is not cohomologhous to a constant, then $e_g(\mu_g^-) <0$.
Thus in our case $\mu_g^-$ is not a  weak limit of measures of type (\ref{rho-n}).
\end{proof}

We show now that the set of maps with negative entropy production
for their respective inverse SRB measures, is \textbf{open and
dense} in a neighbourhood of a hyperbolic toral endomorphism $f$.

\begin{cor}\label{exp}
a) Let $f$ be a hyperbolic toral endomorphism on $\mathbb T^m, m
\ge 2$. Then there exists a neighbourhood $V$ of $f$ in
$\mathcal{C}^1(\mathbb T^m, \mathbb T^m)$ and a set $W \subset V$
such that $W$ is open and dense in the $\mathcal{C}^1$ topology in
$V$ and s.t for any $g \in W$ we have $e_g(\mu_g^-) < 0$.

b) Consider the hyperbolic toral endomorphism on $\mathbb T^2$ given
by  $f(x, y) = (2x+2y, 2x +3y) \ (\text{mod} \ 1)$ and its smooth
perturbation $$g(x, y) = (2x+2y+\vp sin 2\pi y, \ 2x+3y+ 2\vp sin
2\pi y) \ (\text{mod} \ 1)$$ Then the inverse SRB measure of $g$
has negative entropy production, while the SRB measure of $g$ has
positive entropy production, i.e $$ e_g(\mu_g^-) <0 \ \text{and} \
e_g(\mu_g^+) >0$$
\end{cor}

\begin{proof}
a) If $f$ is a hyperbolic toral endomorphism on $\mathbb T^m$ then there exists a neighbourhood $V$ of $f$
in $\mathcal{C}^1$ topology, so that any $g\in V$ is hyperbolic and $d$-to-1, where $d = |\text{det} Df|$.

We showed in Theorem \ref{ent-tor} that $e_g(\mu_g^-) <0$ unless
$|\text{det}Dg|$ is cohomologous to a constant. But from the
Livshitz Theorem (see for instance \cite{KH}) it follows that this
is equivalent to the existence of a constant $c$ such that for any
$n \ge 1$, $$S_n(|\text{det}Dg|)(x) = nc, \ \forall x \in
\text{Fix}(g^n)$$

As the set of $g$'s \textit{not satisfying} the above equalities
is open and dense in $V$, we obtain the conclusion of part a).

b) First of all we notice that $f$ is given by an integer valued
matrix $A$ which has one eigenvalue larger than 1 and another
eigenvalue in $(0, 1)$, so $f$ is hyperbolic. Thus for $\vp>0$
small enough, we have that $g$ (which is well defined as an
endomorphism on $\mathbb T^m$) is hyperbolic as well.

We calculate now the determinant of the derivative of $g$ as
$$\text{det} Dg (x, y) = 2 + 4\pi \vp cos 2 \pi y$$ Now, from
Theorem \ref{ent-tor} we see that $e_g(\mu_g^-) <0$ if and only if
the function $|\text{det} Dg|$ is cohomologous to a constant. But
this is equivalent from the Livshitz conditions (\cite{KH}) to the
fact that there exists a constant $c$ such that
$$S_n(|\text{det}Dg|)(x) = nc, \ x \in \text{Fix}(g^n), n \ge 1$$

In our case, notice that both $(0, 0)$ and $(0, \frac 12)$ are
fixed points for $g$. But $|\text{det}Dg(0, 0)| = 2+4 \pi \vp$,
whereas $|\text{det}Dg(0, \frac 12)| = 2 - 4 \pi \vp$. So the
Livshitz condition above is not satisfied, and $|\text{det}Dg|$ is
not cohomologous to a constant. Hence according to Theorem
\ref{ent-tor} we obtain $$e_g(\mu_g^-) < 0 \ \text{and} \
e_g(\mu_g^+) >0$$

\end{proof}

\textbf{Acknowledgements:} This work was supported by CNCSIS - UEFISCDI, project PNII - IDEI 1191/2008.

\

\textbf{E-mail:}  Eugen.Mihailescu\@@imar.ro

Institute of Mathematics of the Romanian Academy, P. O. Box 1-764,
RO 014700, Bucharest, Romania.

Webpage: www.imar.ro/$\sim$mihailes


\begin{thebibliography}{99}

\bibitem{Bot}
H. G. \ Bothe, Shift spaces and attractors in noninvertible
horseshoes, Fundamenta Math., \textbf{152}, no. 3, 1997, 267-289.

%\bibitem{BPS}
%L. \ Barreira, Y. \ Pesin and J.\ Schmeling, Dimension and product
%structure of hyperbolic measures, Ann. of Math. \textbf{149},
%1999, 755-783.

\bibitem{Bo}
R. \ Bowen, Equilibrium states and the ergodic theory of Anosov
diffeomorphisms, Lecture Notes in Mathematics, 470, Springer 1975.

\bibitem{DSS}
R. L. \ Dobrushin, Ya. G. \ Sinai, Yu. M. \ Sukhov, Dynamical Systems of Statistical Mechanics, in Dynamical Systems, Ergodic Theory and Applications, ed.  Ya. G. \ Sinai, vol. 100, Encyclopaedia of Mathematical Sciences, Springer, 2000.

\bibitem{ER}
J. P.\ Eckmann and D. \ Ruelle, Ergodic theory of strange
attractors, Rev. Mod. Physics, \textbf{57}, 1985, 617-656.

\bibitem{ES}
D. \ Evans and D. \ Searles, The fluctuation theorem, Adv. in Physics, \textbf{51}, no. 7, 2002, 1529-1585.

%\bibitem{ES-94}
%D. \ Evans and D. \ Searles, Equilibrium microstates which generate second law %violating steady states,
%Physical Review E, \textbf{50}, no. 2, 1994, 1645-%1648.

%\bibitem{Ga}
%G. \ Gallavotti, Entropy production in nonequilibrium
%thermodynamics: a review, Chaos, \textbf{14}, 2004, 680-690.

%\bibitem{GC}
%G. \ Gallavotti and E.G.D Cohen, Dynamical ensembles in stationary
%states, J. Stat. Physics \textbf{80}, 1995, 931-970.

\bibitem{KH}
A.\ Katok and B.\ Hasselblatt, Introduction to the Modern Theory
of Dynamical Systems, Cambridge Univ. Press, London-New York,
1995.

\bibitem{Leb}
J. \ Lebowitz, Boltzmann's entropy and time's arrow, Physics Today, \textbf{46}, 1993, 32-38.

\bibitem{L}
F. \ Ledrappier, Proprietes ergodiques des mesures de Sinai, Publ. Math. IHES, vol. 59, 1984, 163-188.

\bibitem{PDL}
P. D \ Liu, Invariant measures satisfying an equality relating
entropy, folding entropy and negative Lyapunov exponents, Commun.
Math. Physics, vol. 284, no. 2, 2008, 391-406.

\bibitem{M-MZ}
E. \ Mihailescu, Unstable directions and fractal dimension for a
class of skew products with overlaps in fibers, Math. Zeitschrift
2010, DOI: 10.1007/s00209-010-0761-y.

\bibitem{M-JSP}
E. \ Mihailescu, Physical measures for multivalued inverse
iterates near hyperbolic repellors, J. Statistical Physics,
\textbf{139}, 2010, 800-819.

\bibitem{M-Cam}
E. \ Mihailescu, Metric properties of some fractal sets and
applications of inverse pressure, Math. Proceed. Cambridge,
\textbf{148}, 3, 2010, 553-572.

\bibitem{M-DCDS06}
E. \ Mihailescu, Unstable manifolds and Holder structures
associated with noninvertible maps, Discrete and Cont. Dynam.
Syst. \textbf{14}, 3, 2006, 419-446.

%\bibitem{MU-BLMS}
%E. \ Mihailescu and M. \ Urbanski, Relations between stable
%dimension and the preimage counting function on basic sets with
%overlaps, Bull. London Math. Soc., \textbf{42}, 2010, 15-27.

%\bibitem{MU-CJM}
%E. \ Mihailescu and M. \ Urbanski, Inverse pressure and the
%independence of stable dimension for noninvertible maps, Canadian
%J. Math. \textbf{60}, no. 3, 2008, 658-684.

%\bibitem{OW}
%D. \ Ornstein and B. \ Weiss, Statistical properties of chaotic
%systems, Bull. AMS, \textbf{24}, no.1, 1991, 11-116.

\bibitem{Pa}
W. \ Parry, Entropy and generators in ergodic theory, W. A
Benjamin, New York, 1969.

\bibitem{QS}
M. \ Qian and Z. Shu, SRB measures and Pesin's entropy formula for
endomorphisms, Trans. Amer. Math. Soc., \textbf{354}, 2002,
1453-1471.

\bibitem{QZ}
M. \ Qian, Z. \ Zhang, Ergodic theory for axiom A endomorphisms,
Ergodic Th. and Dynam. Syst., \textbf{15}, 1995, 161-174.

\bibitem{Ro}
V.\ A.\ Rokhlin, Lectures on the theory of entropy of
transformations with invariant measures, Russian Math. Surveys,
\textbf{22}, 1967, 1-54.

\bibitem{Ru-survey99}
D. \ Ruelle, Smooth dynamics and new theoretical ideas in
nonequilibrium statistical mechanics, J. Statistical Physics
\textbf{95}, 1999, 393-468.

\bibitem{Ru-97}
D. \ Ruelle, Entropy production in nonequilibrium statistical
mechanics, Commun. Math. Phys., \textbf{189}, 1997, 365-371.

\bibitem{Ru-folding}
D. \ Ruelle, Positivity of entropy production in nonequilibrium
statistical mechanics, J. Statistical Physics \textbf{85}, 1/2,
1996, 1-23.

\bibitem{Ru-carte89}
D. \ Ruelle, Elements of differentiable dynamics and bifurcation
theory, Academic Press, New York, 1989.

%\bibitem{Ru-76}
%D. \ Ruelle, A measure associated with Axiom A attractors, Amer. J. Math., \textbf{98}, 1976, 619-654.

\bibitem{S}
Y. \ Sinai, Gibbs measures in ergodic theory, Russian Math.
Surveys, \textbf{27}, 1972, 21-69.

\bibitem{W}
G. M. \ Wang, E. M. \ Sevick, E. \  Mittag, D. J. \ Searles, and
D. J. Evans, Experimental demonstration of violations of the
Second Law of Thermodynamics for small systems and short time
scales, Physical Rev. Letters, \textbf{89}, 050601, 2002.

%\bibitem{Wa}
%P. \ Walters, An introduction to ergodic theory (2nd edition),
%Springer New York, 2000.

\end{thebibliography}
\end{document}